\documentclass[12pt,reqno]{amsart}
\usepackage{geometry}                
\usepackage{amsmath,amssymb}
\usepackage{graphicx}
\usepackage{amssymb,latexsym, amsmath}
\usepackage{amsfonts,amsthm}
\usepackage{amscd}
\usepackage{mathrsfs}
\usepackage{hyperref}
\usepackage{eucal}
\usepackage{epstopdf}
\usepackage{amsgen}
\usepackage{xspace}
\usepackage{verbatim}
\usepackage{stmaryrd}
\usepackage{enumitem}
\newlist{steps}{enumerate}{1}
\setlist[steps, 1]{label = Step \arabic*:}

\newcommand{\gd}{\Delta}

\newcommand{\tu}{\Tilde{u}}
\newcommand{\tv}{\Tilde{v}}
\newcommand{\tw}{\Tilde{w}}

\newcommand{\tg}{\Tilde{g}}

\newcommand{\inpt}[1]{\langle #1 \rangle}

\newcommand{\msa}{\mathcal{A}}

\newcommand{\gw}{\Omega}

\newcommand{\ga}{\gamma}
\newcommand{\gb}{\beta}

\newcommand{\gs}{\sigma}

\newcommand{\nb}{\nabla}
\newcommand{\vp}{\varphi}
\newcommand{\ve}{\varepsilon}
\newcommand{\pdr}{\partial}

\newcommand{\tup}{\textup}

\newcommand{\beq}{\begin{equation}}
\newcommand{\eeq}{\end{equation}}
\newcommand{\bea}{\begin{align}}
\newcommand{\eea}{\end{align}}
\newcommand{\bthm}{\begin{theorem}}
\newcommand{\ethm}{\end{theorem}}
\newcommand{\bpr}{\begin{proof}}
\newcommand{\epr}{\end{proof}}
\newcommand{\bcl}{\begin{corollary}}
\newcommand{\ecl}{\end{corollary}}
\newcommand{\bpn}{\begin{proposition}}
\newcommand{\epn}{\end{proposition}}
\newcommand{\bre}{\begin{remark}}
\newcommand{\ere}{\end{remark}}
\newcommand{\bdf}{\begin{definition}}
\newcommand{\edf}{\end{definition}}
\newcommand{\bss}{\begin{align*}}
\newcommand{\ess}{\end{align*}}

\newcommand{\bl}{\label}

\newtheorem{theorem}{Theorem}[section]
\newtheorem{corollary}[theorem]{Corollary}

\newtheorem{lemma}[theorem]{Lemma}
\newtheorem{proposition}[theorem]{Proposition}

\theoremstyle{definition}
\newtheorem{definition}[theorem]{Definition}
\theoremstyle{remark}
\newtheorem{remark}{Remark}

\numberwithin{equation}{section}

\date{September 7, 2019}                                            

\begin{document}

\title[Nonautonomous Hindmarsh-Rose Equations]{Global Dynamics of Nonautonomous Hindmarsh-Rose Equations}

\author[C. Phan]{Chi Phan}
\address{Department of Mathematics and Statistics, University of South Florida, Tampa, FL 33620, USA}
\email{chi@mail.usf.edu}
\thanks{}

\author[Y. You]{Yuncheng You}
\address{Department of Mathematics and Statistics, University of South Florida, Tampa, FL 33620, USA}
\email{you@mail.usf.edu}
\thanks{}






\begin{abstract}
Global dynamics of nonautonomous diffusive Hindmarsh-Rose equations on a three-dimensional bounded domain in neurodynamics is investigated. The existence of a pullback attractor is proved through uniform estimates showing the pullback dissipative property and the pullback asymptotical compactness. Then the existence of pullback exponential attractor is also established by proving the smoothing Lipschitz continuity in a long run of the solution process.
\end{abstract}

\maketitle

\section{\textbf{Introduction}}

The Hindmarsh-Rose equations for neuronal bursting of the intracellular membrane potential observed in experiments was initially proposed in \cite{HR1, HR2}. The original model is composed of three ordinary differential equations and has been studied by numerical simulations and mathematical analysis, cf. \cite{HR1, HR2, IG, Ko, MFL, SC, SPH, Tr, Su, ZZL} and the references therein. The solutions of this model exhibit interesting bursting patterns, especially chaotic bursting and dynamics. 

Very recently, the authors in \cite{PYS} and \cite{Phan} proved the existence of global attractors for the diffusive and partly diffusive Hindmarsh-Rose equations as well as the existence of a random attractor for the stochastic Hindmarsh-Rose equations with multiplicative noise. 

In this work, we shall study  the global dynamics for the nonautonomous diffusive Hindmarsh-Rose equations with time-dependent external inputs:
\begin{align}
    \frac{\pdr u}{\pdr t} & = d_1 \gd u +  \vp (u) + v - w + J + p_1 (t, x)\, \bl{uq} \\
    \frac{\pdr v}{\pdr t} & = d_2 \gd v + \psi (u) - v + p_2 (t, x), \bl{vq} \\
    \frac{\pdr w}{\pdr t} & = d_3 \gd w + q (u - c) - rw + p_3 (t, x), \bl{wq}
\end{align}
for $t > \tau \in \mathbb{R},\; x \in \gw \subset \mathbb{R}^{n}$ ($n \leq 3$), where $\gw$ is a bounded domain with locally Lipschitz continuous boundary, the stimulation inject current $J $ is assumed to be a constant, and the nonlinear terms in \eqref{uq} and \eqref{vq} are
\beq \bl{pp}
	\vp (u) = au^2 - bu^3, \quad \text{and} \quad \psi (u) = \alpha - \beta u^2.
\eeq 
Assume that the external input terms $p_i \in L^2_{loc} (\mathbb{R}, L^2 (\gw)), i = 1, 2, 3$, satisfy the condition of translation boundedness \cite{CV},
\beq \bl{tbd}
	\|p_i \|^2_{L_b^2} = \sup_{t\, \in \,\mathbb{R}}\, \int_t^{t+1} \int_\gw |p_i (t, x)|^2 \, dx\, ds < \infty, \quad i = 1, 2, 3.
\eeq

All the involved parameters $d_1, d_2, d_3, a, b, \ga, \gb, q, r$, and $J$ are positive constants except $c \,(= u_R) \in \mathbb{R}$, which is a reference value of the membrane potential of a neuron cell. We impose the homogeneous Neumann boundary conditions 
\begin{equation} \label{nbc}
    \frac{\pdr u}{\pdr \nu} (t, x) = 0, \; \; \frac{\pdr v}{\pdr \nu} (t, x)= 0, \; \; \frac{\pdr w}{\pdr \nu} (t, x)= 0,\quad  t > \tau \in \mathbb{R},  \; x \in \partial \gw ,
\end{equation}
and the initial conditions to be specified are denoted by
\begin{equation} \bl{inc}
    u(\tau, x) = u_\tau (x), \; v(\tau, x) = v_\tau (x), \; w(\tau, x) = w_\tau (x), \quad x \in \gw.
\end{equation}

In this system \eqref{uq}-\eqref{wq}, the variable $u(t,x)$ refers to the membrane electric potential of a neuronal cell, the variable $v(t, x)$ represents the transport rate of the ions of sodium and potassium through the fast ion channels and is called the spiking variable, while the variables $w(t, x)$ represents the transport rate across the neuronal cell membrane through slow channels of calcium and other ions correlated to the bursting phenomenon and is called the bursting variable. 

We start with formulation of the aforementioned initial-boundary value problem of \eqref{uq}--\eqref{inc}. Define the Hilbert spaces $H = [L^2 (\gw)]^3 = L^2 (\gw, \mathbb{R}^3)$ and  $E =  [H^{1}(\gw)]^3 = H^1 (\gw, \mathbb{R}^3)$. The norm and inner-product of $H$ or $L^2 (\gw)$ will be denoted by $\| \, \cdot \, \|$ and $\inpt{\,\cdot , \cdot\,}$, respectively. The norm of $E$ will be denoted by $\| \, \cdot \, \|_E$. The norm of $L^p (\gw)$ or $L^p (\gw, \mathbb{R}^3)$ will be dented by $\| \cdot \|_{L^p}$ if $p \neq 2$. We use $| \, \cdot \, |$ to denote a vector norm in a Euclidean space.

The nonautonomous system \eqref{uq}-\eqref{wq} with the initial-boundary conditions \eqref{nbc} and \eqref{inc} can be written in the vector form
\beq \bl{napb}
    \begin{split}
    \frac{\partial g}{\partial t} = Ag &+ f(g) + p(t,x), \quad t > \tau \in \mathbb{R},  \\[3pt]
     &g(\tau) = g_\tau,
   \end{split}
\eeq
where 
$$
	g(t) = \textup{col} \, (u(t, \cdot), v(t, \cdot), w(t, \cdot)), \quad  g_\tau = \textup{col}\, (u_\tau,  v_\tau,  w_\tau),
$$
and $p(t, x) = \textup{col}\, (p_1 (t,x), p_2 (t,x), p(t,x))$, the nonpositive self-adjoint operator
\begin{equation} \label{opA}
        A =
        \begin{pmatrix}
            d_1 \gd  & 0   & 0 \\[3pt]
            0 & d_2 \gd  & 0 \\[3pt]
            0 & 0 & d_3 \gd
        \end{pmatrix}
        : D(A) \rightarrow H,
\end{equation}
where 
$$
	D(A) = \{g \in H^2(\gw, \mathbb{R}^3): \pdr g /\pdr \nu = 0 \},
$$ 
is the generator of an analytic $C_0$-semigroup $\{e^{At}\}_{t \geq 0}$ on the Hilbert space $H$ \cite{SY}. By the fact that $H^{1}(\gw) \hookrightarrow L^6(\gw)$ is a continuous imbedding for space dimension $n \leq 3$ and by the H\"{o}lder inequality, there is a constant $C_0 > 0$ such that 
$$
    \| \vp (u)  \| \leq C_0 \| u \|_{L^6}^3 \quad \tup{and} \quad \|\psi (u) \| \leq C_0 \| u \|_{L^4}^2 \quad \textup{for} \; u \in L^6 (\gw).
$$
Therefore, the nonlinear mapping 
\begin{equation} \label{opf}
    f(u,v, w) =
        \begin{pmatrix}
             \vp (u) + v - w + J \\[4pt]
            \psi (u) - v,  \\[4pt]
	     q (u - c) - rw
        \end{pmatrix}
        : E \longrightarrow H
\end{equation}
is a locally Lipschitz continuous mapping. 

\subsection{\textbf{Hindmarsh-Rose Models in Neurodynamics}}

In 1982-1984, Hindmarsh and Rose developed the mathematical model \cite{HR1, HR2} to describe neuronal dynamics
\begin{equation} \label{HR}
	\begin{split}
    \frac{du}{dt} & = au^2 - bu^3 + v - w + J,  \\
    \frac{dv}{dt} & = \alpha - \beta u^2  - v,  \\
    \frac{dw}{dt} & =  q (u - u_R) - rw.
    \end{split}
\end{equation}
This model characterizes the phenomena of synaptic bursting and especially chaotic bursting. Neuronal signals are short electrical pulses called spike or action potential. Bursting shows alternating phases of rapid firing spikes and then quiescence. It is a mechanism to modulate brain functionalities and to communicate signals with the neighbor neurons. Bursting patterns occur in a variety of bio-systems such as pituitary melanotropic gland, thalamic neurons, respiratory pacemaker neurons, and insulin-secreting pancreatic $\beta$-cells, cf. \cite{BRS, CK,CS, HR2}.

The mathematical analysis mainly using bifurcations together with numerical simulations of several models in ODEs on bursting behavior has been done by many authors, cf. \cite{BB, ET, EI, MFL, Ri, SPH, Tr, WS, Su}. 

Neurons burst through synaptic coupling or diffusive coupling. Synaptic coupling has to reach certain threshold for release of quantal vesicles and synchronization \cite{DJ, Ru, Rv, SC}.  

It is known that Hodgkin-Huxley equations \cite{HH} (1952) provided a four-dimensional model for the dynamics of membrane potential taking into account of the sodium, potassium and leak ions current. It is a highly nonlinear system if without simplification assumptions. The FitzHugh-Nagumo equations \cite{FH} (1961-1962) is a two-dimensional model for an excitable neuron with the membrane potential and the current variable in a lump. This 2D model admits an exquisite phase plane analysis, but it excludes any chaotic solutions and chaotic dynamics so that no chaotic bursting can be generated by the FitzHugh-Nagumo equations. 

In contrast, the Hindmarsh-Rose equations contribute a three-dimensional model with cubic nonlinearity to generate a significant mechanism for rapid firing and busting in the research of neurodynamics. The chaotic coupling exhibited in the simulations and analysis of this Hindmarsh-Rose model shows more rapid synchronization and more effective regularization of neurons due to \emph{lower threshold} than the synaptic coupling \cite{IG, Tr, Rv, SK, SPH, Su}. The research on this Hindmarsh-Rose model also indicated \cite{SC} that it allows for spikes with varying interspike-interval. Therefore, this 3D model is a suitable choice for the investigation of both the regular bursting and the chaotic bursting when the parameters vary. 

In general, neurons are immersed in aqueous biochemical solutions consisting of different ions electrically charged. The axon of a neuron is a long branch to propagate signals and the neuron cell membrane is the conductor along which the voltage signals travel. As pointed out in \cite{EI}, neuron is a distributed dynamical system. 

From physical and mathematical point of view, it is reasonable and useful to consider the diffusive Hindmarsh-Rose model in terms of partial differential equations with the spatial variables $x$ involved, at least in $\mathbb{R}^1$.  Here in the abstract extent, we shall study the diffusive Hindmarsh-Rose equations \eqref{uq}-\eqref{wq} with time-dependent external stimulations in a bounded domain of space $\mathbb{R}^3$ and we shall focus on the global dynamics of the solution processes. The chaotic bursting and dynamical properties from the nonautonomous diffusive Hindmarsh-Rose equations  are expected to demonstrate a wide range of applications in neuroscience. 

\subsection{\textbf{Preliminaries}}

In this work we shall consider the weak solutions of this initial value problem \eqref{napb}. 
\begin{definition} \label{D:wksn}
	A function $g(t, x), (t, x) \in [\tau, T] \times \gw$, is called a weak solution to the initial value problem \eqref{napb}, if the following conditions are satisfied:
	
	\textup{(i)} $\frac{d}{dt} (g, \zeta) = (Ag, \zeta) + (f(g), \zeta)$ is satisfied for a.e. $t \in [\tau, T]$ and for any $\zeta \in E$.
	
	\textup{(ii)} $g(t, \cdot) \in C ([\tau, T]; H) \cap L^2 ([\tau, T], E)$ and $g(\tau ) = g_0$.
	
\noindent
Here $(\cdot , \cdot)$ stands for the dual product of the dual space $E^*$ and $E$.
\end{definition}

\begin{lemma} \label{Lwn}
	For any given initial data $g_0 \in H$, there exists a unique local weak solution $g(t, g_0) = (u(t), v(t), w(t)), \, t \in [\tau, T]$ for some $T > 0$, of the initial value problem \eqref{napb}, such that
\begin{equation} \label{soln}
    g \in C([\tau, T_{max}); H) \cap L_{loc}^2 ([\tau, T_{max}); E),
\end{equation}
where $I_{max} = [\tau, T_{max})$ is the maximal interval of existence. And the weak solution becomes a strong solution on $(\tau, T_{max})$, which satisfies the evolutionary equation \eqref{napb} in $H$ almost everywhere and with the regularity 
\beq \bl{Strs}
    g \in C((\tau, T_{max}); E) \cap C^1 ((\tau, T_{max}); H) \cap L_{loc}^2 ((\tau, T_{max}); H^2 (\gw, \mathbb{R}^3)).
\eeq
\end{lemma}

\begin{proof}
The proof of the local existence and uniqueness of weak solutions is made by \emph{a priori} estimates on the Galerkin approximate solutions obtained by spectral projections of the initial value problem \eqref{napb}, similar to what we shall present in Section 3, and then by the Lions-Magenes type of weak and weak$^*$ compactness and convergence argument. It is an adaptation of the treatment for local solutions of the generic reaction-diffusion system in \cite[Chapter XV, Theorem 3.1 and Proposition 3.1]{CV}. The details are omitted here.
\end{proof}

The Gagliardo-Nirenberg inequalities \cite[Appendix B]{SY} of interpolation shown below will be used in several sharp estimates of this work,
\beq \label{GN}
	\| y \|_{W^{k, p} (\gw)} \leq C \| y \|^\theta_{W^{m, q} (\gw)}\,  \| y \|^{1 - \theta}_{L^r (\gw)}, \quad \text {for all} \;\; y \in W^{m, q} (\gw),
\eeq
where $C > 0$ is a constant, provided that $p, q, r \geq 1, 0 < \theta < 1$, and
$$
	k - \frac{n}{p}\,  \leq \, \theta \left(m - \frac{n}{q}\right)  - (1 - \theta)\, \frac{n}{r},  \quad n = \text{dim}\, (\gw).
$$

In Section 2, we shall recall the basic concepts and the relevant existing results on the topics of global dynamics for nonautonomous dynamical systems. In Section 3, we prove the existence of a pullback attractor for the solution process of the nonautonomous Hindmarsh-Rose equations. In Section 4, the existence of pullback exponential attractors will be proved for this nonautonomous Hindmarsh-Rose process. 

\section{\textbf{Pullback Attractor and Pullback Exponential Attractor}}	

\vspace{2pt}
We refer to \cite{CCLF, CLR1, CLR2, CE1, CE2, EYY, Kl, LMR, LWZ, ZZX} for the concepts and some of the existing results in the theory of nonautonomous dynamical systems, especially on the topics of pullback attractors and pullback exponential attractors. Recall that these concepts are rooted in the theory of global attractors and other invariant attracting sets for the autonomous infinite-dimensional dynamical systems \cite{CV, Milani, Rb, SY, Tm, Y08, Y10, Y12} and the theory of exponential attractors or sometimes called inertial sets \cite{Eden, EMZ, Li, Milani, Yagi}.

Let $X$ be a Banach space and suppose that a nonautonomous partial differential equations with initial-boundary conditions, which usually involves a time-dependent forcing term, has global solutions in space-time. Then the solution operator
$$
	\{S(t, \tau): X \to X\}_{t \geq \tau \in \mathbb{R}}
$$
is called a \emph{nonautonomous process} \cite{CLR2, CV}), which satisfies the three conditions:

1 ) $S(\tau, \tau) = I $ (the identity) for any $\tau \in \mathbb{R}$.

2) The cocycle property is satisfied: 
$$
	S(t, s) S(s, \tau) = S(t, \tau) \quad  \text{for any} \;\; -\infty < \tau \leq s \leq t < \infty.
$$

3) The mapping $(t, \tau, g) \to S(t, \tau) g \in X$ is continuous with respect to $(t, \tau, g) \in \mathcal{T} \times X$ for any given $\tau \in \mathbb{R}$, where $\mathcal{T} = \{(t, \tau) \in \mathbb{R}^2: t \geq \tau\}$.

\begin{definition}[Nonautonomous semiflow] \bl{nas}
	A mapping $\Phi (t, \tau, g): \mathbb{R}^+ \times \mathbb{R} \times X \to X$ is called a \emph{nonautonomous semiflow} (or called nonautonomous dynamical system) on a Banach space $X$ over $\mathbb{R}$, if the following conditions are satisfied:
	
	1) $\Phi (0, \tau, \cdot)$ is the identity on $X$, for any $\tau \in \mathbb{R}$.
	
	2) $\Phi (t + s, \tau, \cdot) = \Phi (t, \tau + s, \Phi(s, \tau, \cdot))$, for any $t, s \geq 0$ and $\tau \in \mathbb{R}$.
	
	3) $\Phi (t, \tau, g): \mathcal{T} \times X \to X$ is continuous.
\end{definition}

If $\{S(t, \tau): X \to X\}_{(t, \tau) \in \mathcal{T}}$ is a continuous evolution process on $X$, then it generates a nonautonomous semiflow defined by 
\beq \bl{PhiS}
	\Phi (t, \tau, g) = S(t + \tau, \tau, g),  \quad (t, \tau, g) \in \mathcal{T} \times X.
\eeq
This relation in the pullback sense is the following important identity
\beq \bl{PSR}
	\Phi (t, \tau - t, g) = S(\tau, \tau - t)g, \quad (t, \tau, g) \in \mathbb{R}^+ \times \mathbb{R} \times X.
\eeq

\begin{definition}[Pullback Attractor] \bl{PA}
      A time-parametrized set $\mathcal{A} = \{\mathcal{A}(\tau)\}_{\tau \in \mathbb{R}}$ in a Banach space $X$ is called a pullback attractor for the nonautonomous semiflow $\{\Phi (t, \tau, \cdot)\}_{(t, \tau) \in \mathcal{T}}$ generated by a continuous evolution process $\{S(t, \tau): X \to X\}_{(t, \tau) \in \mathcal{T}}$, if the following conditions are satisfied:
      
      1) $\mathcal{A}$ is compact in the sense that for each $\tau \in \mathbb{R}$ the set $\mathcal{A}(\tau)$ is compact in $X$.
      
      2) $\mathcal{A}$ is invariant,
$$
	S (t, \tau) \, \mathcal{A} (\tau) = \mathcal{A}(t), \quad t \geq 0, \; \tau \in \mathbb{R}.
$$
it is equivalent to $\Phi (t, \tau, \mathcal{A}(\tau))  = \mathcal{A} (t + \tau)$ for $t \geq \tau$.

      3) $\mathcal{A}$ pullback attracts every bounded set $B \subset X$ with respect to the semi-Hausdorff distance,
$$
	\lim_{t \to \infty} dist_X (\Phi (t, \, \tau - t, B), \, \mathcal{A} (\tau)) = \lim_{t \to \infty} dist_X (S (\tau, \, \tau - t) B, \, \mathcal{A} (\tau)) = 0.
$$
\end{definition}

\begin{definition}[Pullback Exponential Attractor] \label{PEA}
      A time-parametrized set $\mathscr{M} = \{\mathscr{M} (t)\}_{t \in \mathbb{R}} \subset X$, where $X$ is a Banach space, is called a pullback exponential attractor of a continuous evolution process $\{S(t, \tau)\}_{t \geq \tau \in \mathbb{R}}$ on $X$, if the following conditions are satisfied:
      
     1) For any $t \in \mathbb{R}$, the set $\mathscr{M} (t)$ is a compact and positively invariant sel in $X$ with respect to this process,
$$
	S(t, \tau) \mathscr{M}(\tau) \subset \mathscr{M}(t) \quad \text{for any}\;\; \infty < \tau \leq t < \infty.
$$

     2) The fractal dimension $\text{dim}_{F} \mathscr{M}(t)$ for all $t \in \mathbb{R}$ is finite and
$$
     	\sup_{t \in \mathbb{R}} \, \text{dim}_{F} \mathscr{M}(t) < \infty.
$$

     3) $\mathscr{M} = \{\mathscr{M} (t)\}_{t \in \mathbb{R}}$ exponentially attracts every bounded set $B \subset X$ in the sense that there exists a constant rate $\sigma > 0$, a constant $T_B > 0$ depending on $B$, and a positive function $C(\|B\|, T_B)$ where $\|B\| = \sup_{x \in B} \|x\|$, such that
$$
       	\text{dist}_X (S(\tau, \tau - t)B, \mathscr{M}(\tau)) \leq C(\|B\|, T_B) e^{-\sigma (t - \tau)} \quad \text{for any} \; t > T_B, \; \tau \in \mathbb{R}.    
$$
\end{definition}

Below we present two existing results on the sufficient conditions for the existence of pullback attractor and for the existence of pullback exponential attractor, respectively.

\begin{proposition}\cite{CCLF, CLR1, CLR2, Kl} \bl{PA}
      A nonautonomous process $\{S(t, \tau)\}_{t \geq \tau \in \mathbb{R}}$ on a Banach space $X$ has a unique pullback attractor $\msa = \{\msa (\tau)\}_{\tau \in \mathbb{R}}$, if the following two conditions are satisfied\textup{:}
	
 \textup{(i)} There is a pullback absorbing set $M$ in $X$, which means that for any given bounded set $B \subset X$, there is a finite time $T_B > 0$ such that 
 \beq \bl{pbab}
 	S(\tau, \tau - t)B \subset M, \quad \textup{for all} \;\; t > T_B.
 \eeq
 
 \textup{(ii)} The nonautonomous process $S(t, \tau)$ is pullback asymptotically compact in the sense that for any sequences $t_k \to \infty$ and $\{x_k\} \subset B$, where $B$ is any given bounded set in $X$,
the sequence $\{ S(\tau, \tau - t_k)  x_k)\}_{k=1}^\infty$ has a convergent subsequence. Moreover, the pullback attractor is given by
\beq \bl{pbat}
	\msa (\tau) = \bigcap_{s \geq 0} \, \overline{\bigcup_{t \geq s} S (\tau, \tau - t) M}.
\eeq	
\end{proposition}

\begin{proposition} \cite{CE1, CE2} \bl{PeA}
	Let $ X$ and $Y$ be Banach spaces and $Y$ compactly embedded in $X$. Assume that $\{S(t, \tau) \in \mathcal{L} (X) \cap \mathcal{L} (Y): t \geq \tau \in \mathbb{R}\}$ be a nonautonomous process such that the following three conditions are satisfied\textup{:}
	
	\textup{1)} There exists a bounded pullback absorbing set $M^* \subset Y$ uniformly in time in the sense that, for any bounded set $B \subset X$, there is a finite time $T_B > 0$ such that 
\beq \bl{ab}
	\bigcup_{\tau \in \mathbb{R}} S(\tau, \tau - t)B \subset M^*, \quad \textup{for all} \;\;  t > T_B.
\eeq	
	
	\textup{2)} The smoothing Lipschitz continuity is satisfied\textup{:} There is a constant $\kappa > 0$ such that for the aforementioned bounded pullback absorbing set $M^* \subset Y$,
\beq \bl{smL}
	\sup_{\tau \in \mathbb{R}} \|S(\tau, \tau - T_{M^*}) g_1 - S(\tau, \tau - T_{M^*}) g_2 \|_Y \leq \kappa \|g_1 - g_2\|_X, \;\; \text{for any} \;\; g_1, g_2 \in M^*.
\eeq

	\textup{3)} The H\"{o}lder/Lipschitz continuity in time is satisfied\textup{:} There exist two exponents $\gamma_1, \gamma_2 \in (0, 1]$ such that for the aforementioned set $M^* \subset Y$, \begin{align} 
	\sup_{\tau \in \mathbb{R}} &\, \|S(\tau, \tau - T_{M^*}) g - S(\tau, \tau - T_{M^*} - t) g \|_X \leq c_1 |t|^{\gamma_1}, \;\, t \in [0, T_{M^*}], \, g \in M^*, \bl{Lpt1}   \\
	\sup_{\tau \in \mathbb{R}} &\, \|S(\tau, \tau - t_1) g - S(\tau, \tau - t_2) g \|_X \leq c_2 |t_1 - t_2|^{\gamma_2}, \; \,t_1, t_2 \in [T_{M^*},  2T_{M^*}], \,g \in M^*.  \bl{Lpt2}
\end{align}
In \eqref{smL}-\eqref{Lpt2}, $T_{M^*} > 0$ is the time when all the pullback trajectories starting from $M^*$ permanently enter the absorbing set $M^*$ itself, and $c_1 = c_1 (M^*), c_2 = c_2(M^*)$ are two positive constants. Then there exists a pullback exponential attractor $\mathscr{M} = \{\mathscr{M}(\tau)\}_{\tau \in \mathbb{R}}$ in $X$ for this process. 
\end{proposition}

\emph{Remark} 1. The pullback absorbing set can be a time-parametrized set $M(\tau)$ in $X$ or in $Y$. Here the pullback absorbing sets specified in the above Proposition \eqref{PA} and Proposition \eqref{PeA} are time-invariant, which is what we only need.

\emph{Remark} 2. Another concept to describe the asymptotic global dynamics of a nonautonomous PDE is a skew-product dynamical systems \cite{SY}. It is to embed a nonautonomous semiflow into an augmented autonomous semiflow. The corresponding topic is uniform attractor \cite[Chapter IV]{CV}. 

Although a uniform attractor is not a time-parametrized set, the major drawback is that its fractal dimension and Hausdorff dimension of a uniform attractor are in general infinite. The finite dimensionality reduction is lost. Moreover, it is usually difficult to estimate the oftentimes slow rate of attraction for a uniform attractor in terms of physical parameters in the mathematical model. Therefore, pullback attractor and pullback exponential attractor are favorable pursuit of the asymptotic behavior of nonautonomous dynamical systems generated by PDEs.

\section{\textbf{Pullback Attractor for Nonautonomous Hindmarsh-Rose Process}}

In this section, we shall first prove the global existence in time of the weak solutions to the system \eqref{napb} and then show the pullback absorbing property of the nonautonomous Hindmarsh-Rose process in the space $H$ and also in the space $E$, which leads to the existence of a pullback attractor for this nonautonomous semiflow.

\begin{lemma} \bl{naab}
	The weak solution of the nonautonomous system \eqref{napb} for any initial time $\tau \in \mathbb{R}$ and any initial data $g_\tau \in H$ exists globally for $t \in [\tau, \infty)$ and it generates a continuous evolution process $\{S(t,\tau) \in \mathcal{L}(H) \cap \mathcal{L}(E): t \geq \tau \in \mathbb{R}\}$,
\beq \bl{CP}
    S(t,\tau) g_\tau = g(t, \tau, g_\tau) = \textup{col}\, (u, v, w)(t, \tau, g_\tau)
\eeq
which is called the nonautonomous Hindmaersh-Rose process. Moreover, there exists a time-invariant pullback absorbing set in the space $H$,
\beq \bl{naB}
    M^*_H = \{g \in H: \|g \|^2 \leq K_1\}
\eeq
where $K_1$ is a positive constant independent of $\tau$ and $t$ in the sense that for any given bounded set $B \subset H$,
\beq \bl{pkab}
    S(\tau, \tau - t)B \subset M^*_H, \quad \textup{for}\; \; t \geq T_B,
\eeq
where the constant $T_B > 0$ depend only on $\|B\| =\sup_{g \in B} \|g\|$.
\end{lemma}

\begin{proof}
    Take the $H$ inner-product $\langle \eqref{napb}, (c_1 u, v, w) \rangle$ with constant $c_1 > 0$ to obtain
\beq \bl{Sab}
    \begin{split}
    &\frac{1}{2}\frac{d}{dt} \left(c_1 \|u\|^2+ \|v\|^2 + \|w\|^2\right) + \left(c_1 d_1 \|\nb u\|^2 + d_2\|\nb v\|^2 + d_3\|\nb w\|^2\right) \\
    = &\, \int_\gw c_1 (au^3 - bu^4 +uv - uw + Ju + u p_1 (t, x))\,dx\\
    + &\, \int_\gw (\alpha v - \beta u^2 v - v^2 + v p_2(t, x) + q (u - c)w - rw^2 + w p_3 (t, x))\, dx \\
    \leq &\,  \int_\gw c_1 (au^3 - bu^4 +uv - uw + Ju + u p_1 (t, x))\,dx\\
    + &\, \int_\gw \left[\left(2\alpha^2 + \frac{1}{2}\beta^2 u^4 - \frac{3}{8}v^2 \right) + \left(\frac{q^2}{r}(u^2 + c^2) - \frac{1}{2}rw^2\right)+ vp_2 + wp_3\right] dx \\
    \leq &\, \int_\gw c_1 (au^3 - bu^4 +uv - uw + Ju + u p_1 (t, x))\,dx\\
    + &\, \int_\gw \left[\left(2\alpha^2 + \frac{1}{2}\beta^2 u^4 - \frac{3}{8}v^2 \right) + \left(\frac{q^2}{r}(u^2 + c^2) - \frac{1}{2}rw^2\right) \right] dx \\
    +&\,\int_\gw \left[ \frac{1}{8}v^2 + 2|p_2(t,x)|^2 + \frac{1}{8}rw^2 + \frac{2}{r}|p_3(t,x)|^2\right] dx.   
    \end{split}
\eeq
Choose the positive constant in \eqref{Sab} to be $c_1 = \frac{1}{b}(\beta^2 + 3)$ so that 
$$
	- c_1\int_\gw bu^4\, dx + \int_\gw \beta^2 u^4\, dx \leq -3 \int_\gw u^4\, dx.
$$
Note that
$$
	\int_\gw c_1 a u^3\,dx \leq \frac{3}{4}\int_\gw u^4\,dx+\frac{1}{4} (c_1 a)^4 |\gw | \leq \int_\gw u^4 \,dx + (c_1 a)^4 |\gw |,
$$
and
\begin{gather*}
	\int_\gw c_1 (uv - uw + Ju + up_1 (t,x))\,dx \leq \int_\gw \left[2(c_1 u)^2 + \frac{1}{8}v^2 + \frac{(c_1 u)^2}{r}+\frac{1}{4} rw^2 \right. \\
	 \left. +\frac{1}{2}\left((c_1 u)^2 + J^2 +(c_1 u)^2 +|p_1(t,x)|^2\right)\right] dx. 
\end{gather*}
The collection of all integral terms of $u^2$ in the above inequality and in \eqref{Sab} satisfies
$$
	\int_\gw \left(2(c_1u)^2 + \frac{(c_1 u)^2}{r}+ (c_1 u)^2 + \frac{q^2}{r}u^2 \right)dx \leq \int_\gw u^4 \,dx+ \left[c_1^2 \left(3 + \frac{1}{r} \right) + \frac{q^2}{r}\right]^2 |\gw |.
$$
Substitute these inequalities into \eqref{Sab}. Then we get
\beq \bl{SAB}
\begin{split}
	&\frac{1}{2}\frac{d}{dt} \left(c_1 \|u\|^2+ \|v\|^2 + \|w\|^2\right) + \left(c_1 d_1 \|\nb u\|^2 + d_2\|\nb v\|^2 + d_3\|\nb w\|^2\right) \\[3pt]
	\leq &\, \int_\gw c_1 (au^3 - bu^4 +uv - uw + Ju + u p_1 (t, x))\,dx\\
	+ &\, \int_\gw \left[\left(2\alpha^2 + \frac{1}{2}\beta^2 u^4 - \frac{3}{8}v^2 \right) + \left(\frac{q^2}{r}(u^2 + c^2) - \frac{1}{2}rw^2\right) \right] dx \\
	+&\,\int_\gw \left[ \frac{1}{8}v^2 + 2|p_2(t,x)|^2 + \frac{1}{8}rw^2 + \frac{2}{r}|p_3(t,x)|^2\right] dx \\
	\leq &\, \int_\gw (2 - 3)u^4\,dx + \int_\gw\left(\frac{1}{8} - \frac{3}{8} + \frac{1}{8}\right) v^2\,dx + \int_\gw \left( \frac{1}{4} - \frac{1}{2} + \frac{1}{8}\right) rw^2\,dx \\
	+ &\, \int_\gw \left[\frac{1}{2}|p_1 (t,x)|^2+ 2|p_2 (t,x)|^2 + \frac{2}{r}|p_3 (t,x)|^2 \right] dx \\
	+ &\,\left((c_1 a)^4 + J^2 + \left[c_1^2 \left(3 + \frac{1}{r}\right) + \frac{q^2}{r} \right]^2 + 2\alpha^2 + \frac{q^2c^2}{r}\right) |\gw| \\
	\leq &\, - \int_\gw \left(u^4 (t,x) + \frac{1}{8} v^2 (t,x) + \frac{1}{8} rw^2(t,x)\right) dx + \left(2 + \frac{2}{r}\right)\| p(t)\|^2 + c_2 |\gw |,
	\end{split}
\eeq
where
$$
	c_2 = (c_1 a)^4 + J^2 + \left[c_1^2 \left(3 + \frac{1}{r}\right) + \frac{q^2}{r}\right]^2 + 2\alpha^2 + \frac{q^2c^2}{r}.
$$
It follows that
\begin{equation*}
	\begin{split}
	&\frac{d}{dt} (c_1 \|u(t)\|^2 + \|v(t)\|^2 +\|w(t)\|^2) + 2d (c_1 \|\nb u\|^2 + \|\nb v\|^2 + \|\nb w\|^2) \\[5pt]
	+ &\,\int_\gw \left(2u^4 (t, x) + \frac{1}{4}v^2 (t,x) + \frac{1}{4}rw^2(t,x)\right) dx \leq 4 \left(1 + \frac{1}{r}\right)\| p(t)\|^2 + 2c_2 |\gw |,	
	\end{split}
\end{equation*}
where $d = \min \{d_1, d_2, d_3\}$ and we used $$2u^4 \geq \frac{1}{4}\left(c_1 u^2 - \frac{c_1^2}{32}\right).$$
Therefore, 
\begin{equation} \bl{nap}
\begin{split}
	&\frac{d}{dt} (c_1 \|u(t)\|^2 + \|v(t)\|^2 +\|w(t)\|^2) + 2d (c_1 \|\nb u\|^2 + \|\nb v\|^2 + \|\nb w\|^2) \\[8pt]
	+ \frac{1}{4}&\,(c_1 \|u(t)\|^2 + \|v(t)\|^2 + r\|w(t)\|^2)  \leq 4 \left(1 + \frac{1}{r}\right)\| p(t)\|^2 + \left(\frac{c_1^2}{128} + 2c_2\right) |\gw|	
\end{split}
\end{equation}
for $t \in [\tau, T_{max})$, the maximum time interval of existence. Set 
$$
	\delta = \frac{1}{4} \min \{1, r\}.
$$
Then the Gronwall inequality applied to the inequality reduced from \eqref{nap},
\begin{equation*}
    \begin{split}
    &\frac{d}{dt} (c_1 \|u(t)\|^2 + \|v(t)\|^2 +\|w(t)\|^2) + \delta (c_1 \|u(t)\|^2 + \|v(t)\|^2 + \|w(t)\|^2) \\[8pt]
    \leq &\, 4 \left(1 + \frac{1}{r}\right)\| p(t)\|^2 + \left(\frac{c_1^2}{128} + 2c_2\right) |\gw|,	
    \end{split}
\end{equation*}
shows that
\beq \bl{glab}
	\begin{split}
	&c_1 \|u(t)\|^2 + \|v(t)\|^2 +\|w(t)\|^2 \leq e^{- \delta t}(c_1 \|u_\tau \|^2 + \|v_\tau \|^2 +\|w_\tau \|^2) \\[8pt]
	+ 4 &\, \left(1 + \frac{1}{r}\right) \int_\tau^t e^{- \delta (t-s)}\| p(s)\|^2\, ds + \frac{1}{\delta}\left(\frac{c_1^2}{128} + 2c_2\right) |\gw|, \quad t \in [\tau, T_{max}).
	\end{split}
\eeq
By the assumption \eqref{tbd} on the translation boundedness of the external input terms and the upper bound estimate \eqref{glab}, the weak solutions will never blow up at any finite time so that $T_{max} = +\infty$ for all $\tau \in \mathbb{R}$ and any initial data $g_\tau \in H$. Thus the global existence in time of the weak solutions in the space $H$ is proved. Together with the uniqueness and the continuous dependence of $(t, \tau, g_\tau)$ which can be shown, the statement of the continuous evolution process $S(t, \tau)$ in \eqref{CP} is proved. 

In order to prove the claimed existence of a pullback absorbing set, we can exploit the bounded translation property \eqref{tbd} of the time-dependent forcing terms to treat the integral in \eqref{glab} on the time interval $[\tau, t + \tau]$, or equivalently the time interval $[\tau - t, \tau]$, for $t > 0$, as follows:

\beq \bl{pullab}
	\begin{split}
	&c_1\|u(t+\tau)\|^2 + \|v(t+\tau)\|^2 + \|w(t+\tau)\|^2 \\[5pt]
	\leq &\, e^{-\delta t}(c_1\|u(\tau)\|^2 + \|v(\tau)\|^2 + \|w(\tau)\|^2) + \frac{1}{\delta}\left(\frac{c_1^2}{128} + 2c_2\right) |\gw |\\
	+ &\, 4 \left(1 + \frac{1}{r}\right) \int_\tau^{t+\tau} e^{- \delta (t+\tau - s)}\| p(s)\|^2\, ds \\
	\leq &\, e^{-\delta t}(c_1\|u(\tau)\|^2 + \|v(\tau)\|^2 + \|w(\tau)\|^2) + \frac{1}{\delta}\left(\frac{c_1^2}{128} + 2c_2\right) |\gw |\\
	+ &\, 4 \left(1 + \frac{1}{r}\right) \sum_{k = 0}^{\infty} \int_{t+\tau-k-1}^{t+\tau -k} e^{- \delta (t+\tau - s)}\| p(s)\|^2\,ds \\
	\leq &\, e^{-\delta t}(c_1\|u(\tau)\|^2 + \|v(\tau)\|^2 + \|w(\tau)\|^2) + \frac{1}{\delta}\left(\frac{c_1^2}{128} + 2c_2\right) |\gw |\\
	+ &\, 4 \left(1 + \frac{1}{r}\right) \sum_{k = 0}^{\infty} e^{-k \delta}\left(\|p_1 \|^2_{L^2_b} + \|p_2 \|^2_{L^2_b} + \|p_3 \|^2_{L^2_b}\right) \\
	= &\, e^{-\delta t}(c_1\|u_\tau\|^2 + \|v_\tau\|^2 + \|w_\tau\|^2) + \frac{1}{\delta}\left(\frac{c_1^2}{128} + 2c_2\right) |\gw |\\
	+ &\, 4 \left(1 + \frac{1}{r}\right) \frac{1}{1 - e^{-\delta}}\left(\|p_1 \|^2_{L^2_b} + \|p_2 \|^2_{L^2_b} + \|p_3 \|^2_{L^2_b}\right). 
	\end{split}
\eeq
It implies that the global weak solutions of the nonautonomous diffusive Hindmarsh-Rose system \eqref{napb} admit the estimate that, for any $t \geq \tau \in \mathbb{R}$,
\beq \bl{gest}
	\|g(t)\|^2 \leq \frac{\max \{1, c_1\}}{\min \{1,c_1\}} e^{-\delta (t - \tau)}\|g(\tau)\|^2 + \frac{1}{\delta}\left(\frac{c_1^2}{128} + 2c_2\right) |\gw|
	+ 4 \left(1 + \frac{1}{r}\right) \frac{\|p\|^2_{L^2_b}}{1 - e^{-\delta}}. 
\eeq
Hence, for any $\tau - t \leq \tau \in \mathbb{R}$ with $t > 0$, it holds that
\beq \bl{ttau}
	\|g(\tau)\|^2 \leq \frac{\max \{1, c_1\}}{\min \{1,c_1\}} e^{-\delta t}\|g(\tau - t)\|^2 + \frac{1}{\delta}\left(\frac{c_1^2}{128} + 2c_2\right) |\gw|
	+ 4 \left(1 + \frac{1}{r}\right) \frac{\|p\|^2_{L^2_b}}{1 - e^{-\delta}}. 
\eeq
Since 
$$
	\lim_{t \to \infty} e^{-\delta (t - \tau)}\|g(\tau)\|^2 = 0 
$$
uniformly for $g(\tau) = g_\tau$ in any given bounded set $B \subset H$ in regard to \eqref{gest}, and
$$
	\lim_{t \to \infty} e^{-\delta t}\|g(\tau - t)\|^2 = 0
$$
uniformly for $g(\tau - t)$ in any given bounded set $B \subset H$ in regard to \eqref{ttau}, there exists a pullback absorbing set as claimed in \eqref{naB} with the constant
\beq \bl{M1}
	K_1 = 1 +  \frac{1}{\delta}\left(\frac{c_1^2}{128} + 2c_2\right) |\gw|
	+ 4 \left(1 + \frac{1}{r}\right) \frac{\|p\|^2_{L^2_b}}{1 - e^{-\delta}}
\eeq
which is independent of initial time and initial state in $H$. Therefore, the pullback absorbing property \eqref{pkab} for any given bounded set $B \subset H$ is proved: 
$$
	S(\tau, \tau- t)B \subset M^*_H, \quad \textup{for all}\;\; t \geq T_B,
$$
and 
\beq \bl{TB}
	T_B = \frac{1}{\delta}\log^+ \left(\frac{\max \{1, c_1\}}{\min \{1,c_1\}} \|B\|^2 \right) > 0
\eeq	
depends only on $\|B\|$. The proof is completed.
\end{proof}

\begin{lemma} \bl{naac}
	For the nonautonomous diffusive Hindmarsh-Rose system \eqref{napb}, there also exists a time-invariant pullback absorbing set in the space $E$,
\beq \bl{acB}
	M^*_E = \{g \in E: \|g\|^2_E \leq K_2\},
\eeq
where $K_2$ is a positive constant, such that for any given bounded set $B \subset H$,
\beq \bl{naE}
	S(\tau, \tau - t)B \subset M^*_E \quad \textup{for all}\;\;\,  t \geq T_B + 1,
\eeq
for any $\tau \in \mathbb{R}$, where the constant $T_B$ is given in \eqref{TB}.
\end{lemma}

\begin{proof}
	Take the $H$ inner-product $\langle \eqref{napb}, - \Delta g(t) \rangle$ to obtain
\begin{equation*}
	\begin{split}
	&\frac{1}{2}\frac{d}{dt} \left(\|\nb u \|^2 + \|\nb v\|^2 + \|\nb w\|^2\right) + d_1 \|\Delta u\|^2 + d_2 \|\Delta v\|^2 + d_3 \|\Delta w\|^2 \\[3pt]
	=&\,\int_\gw (-au^2 \Delta u - 3bu^2|\nb u\|^2 -v\Delta u + w \Delta u - J\Delta u - p_1(t,x)\Delta u)\, dx \\
	+&\, \int_\gw (-\alpha \Delta v + \beta u^2 \Delta v - |\nb v|^2)\, dx + \int_\gw (qc \Delta w- qu \Delta w - r|\nb w |^2)\, dx \\
	- &\, \int_\gw (p_2(t, x) \gd v + p_3(t, x) \gd w)\,dx.
	\end{split}
\end{equation*}
By using Young's inequality appropriately to treat the integral terms on the right-hand side of the above inequality, we can get 
\beq \bl{acg}
	\begin{split}
	&\frac{d}{dt} \left(\|\nb u \|^2 + \|\nb v\|^2 + \|\nb w\|^2\right) + d_1 \|\Delta u\|^2 + d_2 \|\Delta v\|^2 + d_3 \|\Delta w\|^2 \\[5pt]
	+ &\, 6b \|u \nb u\|^2 + 2\|\nb v \|^2 + 2r \|\nb w \|^2 \\
	\leq &\, \frac{4}{d_1} \|v\|^2 + \frac{4}{d_1}\| w \|^2+ \frac{4a^2}{d_1} \|u\|^4_{L^4} + \frac{8J^2}{d_1} |\gw | + \frac{8}{d_1} \|p_1 (t)\|^2 \\
	+ &\, \frac{2\beta^2}{d_2} \|u \|^4_{L^4} + \frac{4\alpha^2}{d_2} |\gw| + \frac{4}{d_2} \|p_2 (t)\|^2 + \frac{2q^2}{d_3} \|u\|^2 + \frac{4q^2c^2}{d_3} |\gw| +\frac{4}{d_3} \|p_3(t)\|^2 \\
	= &\,  \frac{4}{d_1} \|v\|^2 + \frac{4}{d_1}\| w \|^2  + \frac{2q^2}{d_3} \|u\|^2 +  \left(\frac{4a^2}{d_1} + \frac{2\beta^2}{d_2}\right) \|u\|^4_{L^4} \\
	+&\, \left(\frac{8J^2}{d_1} + \frac{4\alpha^2}{d_2} + \frac{4q^2c^2}{d_3}\right) |\gw| + \frac{8}{d_1} \|p_1 (t)\|^2  + \frac{4}{d_2} \|p_2 (t)\|^2 + \frac{4}{d_3} \|p_3(t)\|^2.
	\end{split}
\eeq
The Sobolev imbedding $H^1(\gw) \hookrightarrow L^4 (\gw)$ tells us that there is a positive constant $\rho > 0$ such that
\beq \bl{L4H1}
	\|u\|^4_{L^4} \leq \rho (\|u\|^2 + \|\nb u\|^2)^2 \leq 2\rho (\|u\|^4 +\|\nb u \|^4).
\eeq
According to Lemma \ref{naab}, for any given bounded set $B \subset H$, we have
\beq \bl{rho}
	\|u(t)\|^2 + \|v(t))\|^2 + \|w(t)\|^2 \leq K_1, \quad \textup{for any}\;\;t \geq T_B, \; g_\tau \in B.
\eeq
Then \eqref{acg} yields the following inequality that for any $t \geq T_B$ and $g_\tau \in B$,
\beq \bl{uniGW}
	\begin{split}
	 &\frac{d}{dt} \left(\|\nb u \|^2 + \|\nb v\|^2 + \|\nb w\|^2\right) + d_1 \|\Delta u\|^2 + d_2 \|\Delta v\|^2 + d_3 \|\Delta w\|^2 \\[8pt]
	+ &\, 6b \|u \nb u\|^2 + 2\|\nb v \|^2 + 2r \|\nb w \|^2 \\[7pt]
	\leq &\, \max \left\{\frac{4}{d_1}, \frac{4q^2 c^2}{d_3}\right\} K_1 +  \left(\frac{8a^2}{d_1} + \frac{4\beta^2}{d_2}\right) \rho K_1^2 + \left(\frac{8a^2}{d_1} + \frac{4\beta^2}{d_2}\right) \rho \|\nb u\|^4 \\
	+&\, \left(\frac{8J^2}{d_1} + \frac{4\alpha^2}{d_2} + \frac{4q^2c^2}{d_3}\right) |\gw| + \frac{8}{d_1} \|p_1 (t)\|^2  + \frac{4}{d_2} \|p_2 (t)\|^2 + \frac{4}{d_3} \|p_3(t)\|^2 .
	\end{split}
\eeq
Hence we can apply the uniform Gronwall inequality \cite[Lemma D.3]{SY} to the following inequality reduced from \eqref{uniGW} on $\nb g(t)= \text{col} \,(\nb u(t), \nb v(t), \nb w(t))$,
\beq \bl{uGW}
	\begin{split}
	&\frac{d}{dt}\|\nb g(t)\|^2 \leq \rho \left(\frac{8a^2}{d_1} + \frac{4\beta^2}{d_2}\right) \|\nb g\|^2 \|\nb g\|^2 \\
	+&\,  \max \left\{\frac{4}{d_1}, \frac{4q^2 c^2}{d_3}\right\} K_1 +  \left(\frac{8a^2}{d_1} + \frac{4\beta^2}{d_2}\right) \rho K_1^2 \\
	+&\,\left(\frac{8J^2}{d_1} + \frac{4\alpha^2}{d_2} + \frac{4q^2c^2}{d_3}\right) |\gw| + \frac{8}{d_1} \|p_1 (t)\|^2  + \frac{4}{d_2} \|p_2 (t)\|^2 + \frac{4}{d_3} \|p_3(t)\|^2
	\end{split}
\eeq
which is written in the form 
\beq \bl{rs}
	\frac{d\gs}{dt} \leq \xi \,\gs + h, \quad \text{for}\;\;  t \geq T_B, \; g_\tau \in B,
\eeq
where 
\begin{gather*}
	\gs (t) = \|\nb g(t)\|^2,   \quad  \xi (t) = \rho \left(\frac{8a^2}{d_1} + \frac{4\beta^2}{d_2}\right) \|\nb g\|^2,  \quad \text{and}  \\
	h(t) = \max \left\{\frac{4}{d_1}, \frac{4q^2 c^2}{d_3}\right\} K_1 +  \left(\frac{8a^2}{d_1} + \frac{4\beta^2}{d_2}\right) \rho K_1^2 \\
	+ \left(\frac{8J^2}{d_1} + \frac{4\alpha^2}{d_2} + \frac{4q^2c^2}{d_3}\right) |\gw| + \frac{8}{d_1} \|p_1 (t)\|^2  + \frac{4}{d_2} \|p_2 (t)\|^2 + \frac{4}{d_3} \|p_3(t)\|^2.
\end{gather*}
For $t \geq T_B$, by integration of the inequality \eqref{nap} we can deduce that
\begin{gather*}
	\int_t^{t+1} 2d(c_1 \|\nb u(s)\|^2 +\|\nb v(s)\|^2+ \|\nb w(s)\|^2)\, ds\\
	\leq c_1 \|u(t)\|^2 + \|v(t)\|^2 + \|w(t)\|^2 + 4\left(1 + \frac{1}{r}\right) \int_t^{t+1} \|p(s)\|^2 ds + \left(\frac{c_1^2}{128} + 2c_2\right)|\gw | \\
	\leq \max \{1, c_1\} K_1+ 4\left(1 + \frac{1}{r}\right) \| p \|^2_{L^2_b} + \left(\frac{c_1^2}{128} + 2c_2\right)|\gw |,
\end{gather*}
where $ \| p \|^2_{L^2_b} = \sum_{i=1}^3  \| p_i \|^2_{L^2_b}$. Denote by
$$
	N_1 =  \frac{1}{2d \min \{1, c_1\}} \left[\max \{1, c_1\} K_1+ 4\left(1 + \frac{1}{r}\right) \| p \|^2_{L^2_b} + \left(\frac{c_1^2}{128} + 2c_2\right)|\gw |\right]
$$
and
\begin{align*}
	N_2 = \max \left\{\frac{4}{d_1}, \frac{4q^2 c^2}{d_3}\right\} K_1 +  \left(\frac{8a^2}{d_1} + \frac{4\beta^2}{d_2}\right) \rho K_1^2 + \left(\frac{8J^2}{d_1} + \frac{4\alpha^2}{d_2} + \frac{4q^2c^2}{d_3}\right) |\gw|.
\end{align*}
Then we have
\beq \bl{gsphih}
	\begin{split}
	\int_t^{t+1} \gs (s)\,ds &\leq N_1, \\
	\int_t^{t+1} \xi (s)\, ds &\leq \rho \left(\frac{8a^2}{d_1} + \frac{4\beta^2}{d_2}\right)N_1, \\
	\int_t^{t+1} h(s)\,ds & \leq N_2 + \max \left\{\frac{8}{d_1}, \, \frac{4}{d_2}, \,\frac{4}{d_3}\right\}\|p\|^2_{L^2_b}.
	\end{split}
\eeq
Thus the uniform Gronwall inequality applied to \eqref{rs} shows that
\beq \bl{BE}
	\|\nb g(t)\|^2 \leq \left(N_1 + N_2 + \max \left\{\frac{8}{d_1}, \, \frac{4}{d_2}, \,\frac{4}{d_3}\right\}\|p\|^2_{L^2_b} \right) \exp \left\{ \rho \left(\frac{8a^2}{d_1} + \frac{4\beta^2}{d_2}\right)N_1\right\}, 
\eeq
for all $t \geq T_B + 1$ and all $g_\tau \in B$. Therefore, the claim \eqref{acB} of a pullback absorbing ball $M^*_E$ in the space $E$ is proved and the constant $K_2$ is given by
$$
	K_2 = K_1 + \left(N_1 + N_2 + \max \left\{\frac{8}{d_1}, \, \frac{4}{d_2}, \,\frac{4}{d_3}\right\}\|p\|^2_{L^2_b} \right) \exp \left\{ \rho \left(\frac{8a^2}{d_1} + \frac{4\beta^2}{d_2}\right)N_1\right\}.
$$
Indeed, for any given bounded set $B \subset H$, we have
$$
	S(\tau, \tau - t)B \subset M^*_E \quad \text{for all}\;\; t \geq T_B + 1.
$$
The proof is completed.
\end{proof}
Now we prove the first main result of this paper. 
\begin{theorem} \bl{PAC}
	Under the assumption \eqref{tbd}, for any positive parameters and $c \in \mathbb{R}$ in the Hindmarsh-Rose equations \eqref{uq}-\eqref{wq}, there exists a pullback attractor $\mathcal{A} = \{\mathcal{A} (\tau)\}_{\tau \in \mathbb{R}}$ in $H$ for the nonautonomous Hindmarsh-Rose process $\{S(t, \tau)\}_{t \geq \tau \in \mathbb{R}}$.
\end{theorem}

\begin{proof}
	By Lemma \ref{naab}, there exists a pullback absorbing set $M^*_H$ in $H$ for the solution process $\{S(t, \tau): t \geq \tau \in \mathbb{R}\}$ of the nonautononous Hindmarsh-Rose system \eqref{napb} so that the first condition in Proposition \ref{PA} is satisfied. 
	
	By Lemma \ref{naac} and the compact embedding $E \hookrightarrow H$, the existence of a pullback absorbing set $M^*_E$ in $E$ for this nonautonomous process shows that any sequence $\{S(\tau, \tau - t_k)g_k\}_{k=1}^\infty$, where $t_k \to \infty$ and $\{g_k\}$ in any given bounded set of $H$ has a convergent subsequence. 
	Thus the second condition of the pullback asymptotic compactness in Proposition \ref{PA} is also satisfied.
	
	Then by Proposition \ref{PA}, there exists a pullback attractor $\mathcal{A} = \{\mathcal{A}(\tau)\}_{\tau \in \mathbb{R}}$, 
$$
	\mathcal{A}(\tau) = \bigcap_{s \geq 0} \, \overline{\bigcup_{t \geq s} S (\tau, \tau - t)M^*_H} \, ,
$$
for this nonautonomous Hindmarsh-Rose process.  
\end{proof}

\section{\textbf{The Existence of Pullback Exponential Attractor}}

In this section, we shall prove the existence of a pullback exponential attractor for the nonautonomous Hindmarsh-Rose process based on Proposition \ref{PeA}. The key leverage is to prove the smoothing Lipschitz continuity of this nonautonomous process with respect to the initial data. 

\begin{theorem}[Smoothing Lipschitz Continuity] \bl{LpHE}
	For the nonautonomous Hindmarsh-Rose process $\{S(t, \tau)\}_{t \geq \tau \in \mathbb{R}}$ in \eqref{CP} generated by the weak solutions of the nonautonomous Hinsmarsh-Rose system \eqref{napb}, there exists a constant $\kappa > 0$ such that 
\beq \bl{lpH}
	\sup_{\tau \in \mathbb{R}} \|S(\tau, \tau - T_{M^*_E})g_\tau - S(\tau,\tau- T_{M^*_E})\tg_\tau \|_E \leq \kappa \|g_\tau - \tg_\tau\|,  \quad \text{for}\;\, g_\tau, \tg_\tau \in M^*_E ,
\eeq
where $T_{M^*_E} > 0$ is the time when all the pullback solution trajectories of \eqref{napb} starting from the set $M^*_E$ in \eqref{acB} permanently enter the set $M^*_E$ itself shown in Lemma \ref{naac}.
\end{theorem}

\begin{proof}
	It is equivalent to prove that 
\beq \bl{lpHe}
	\sup_{\tau \in \mathbb{R}} \|S(\tau + T_{M^*_E}, \tau)g_\tau - S(\tau + T_{M^*_E}, \tau)\tg_\tau \|_E \leq \kappa \|g_\tau - \tg_\tau\|, \quad g_\tau, \tg_\tau \in M^*_E.
\eeq
Denote two solutions with any given initial data $g_\tau$ and $\tg_\tau$ by
$g(t) = (u(t), v(t), w(t))$ and $\tg (t) = (\tu (t), \tv (t), \tw (t))$, respectively. Denote the difference by $\Pi (t) = g(t) - \tg (t) = (U(t), V(t), W(t))$. Then $\Pi (t)$ is the solution of the following intial value problem
\begin{equation} \bl{vpeq}
	\begin{split}
	\frac{d\Pi}{dt} = &\, A\Pi + f(g) -f(\tg), \quad t \geq \tau \in \mathbb{R}, \\[3pt]
	&\Pi (\tau) = g_\tau - \tg_\tau.	
	\end{split}
\end{equation}
 
Step 1.  Take the inner-product $\langle \eqref{vpeq}, \Pi (t) \rangle$ through three component equations \eqref{uq}-\eqref{wq}. For the first component equation of $\Pi (t) = g(t) - \tg (t)$, we get
\beq \label{Uq}
	\begin{split}
	&\frac{1}{2} \frac{d}{dt} \|U(t)\|^2 + d_1\|\nb U(t)\|^2 =  \langle f_1 (g) - f_1 (\tg), u - \tu \rangle \\[2pt]
	= &\,  \int_\gw \left(a(u -  \tu)^2 (u + \tu) - b(u - \tu)^2 (u^2 \tu + u \tu + \tu^2)\right) dx \\
	&\, + \int_\gw ((v  - \tv)(u - \tu) - (w - \tw)(u - \tu)) \, dx \\
	\leq &\, \int_\gw (u -  \tu)^2 \left[ a(u  + \tu) - b (u^2 + u \tu + \tu^2)\right] dx \\[5pt]
	&\, + \| u - \tu \| (\|v - \tv \| + \|w - \tw \|) \\[5pt]
	\leq &\, \int_\gw (u -  \tu)^2 \left[ a(u  + \tu) - b (u^2 + u \tu + \tu^2)\right] dx + 2 \|g - \tg \|^2
	\end{split}
\eeq
and by Young's inequality we have 
\begin{gather*}
	a (u + \tu) -  - b (u^2 + u \tu + \tu^2) = [a (u + \tu) - b u\tu] - b(u^2 + \tu^2) \\[5pt]
	\leq  \left(\frac{b}{4} u^2 + \frac{a^2}{b}\right) + \left(\frac{b}{4} \tu^2 + \frac{a^2}{b}\right) + \frac{b}{2} (u^2 + \tu^2) - b(u^2 + \tu^2) \leq - \frac{b}{4} (u^2 + \tu^2) + \frac{2a^2}{b}.
\end{gather*}
It follows that
\beq \label{Ue}
	\begin{split}
	& \frac{d}{dt} \|U(t)\|^2 \leq \frac{d}{dt} \|U(t)\|^2 + 2d_1\|\nb U(t)\|^2 \\[5pt]
	\leq &\, 2\int_\gw (u - \tu)^2 \left( - \frac{b}{4} (u^2 + \tu^2) + \frac{2a^2}{b}\right) dx + 4 \|g - \tg\|^2 \\
	\leq &\, \int_\gw (u - \tu)^2 \left( - \frac{b}{2} (u^2 + \tu^2)\right) dx + \frac{4a^2}{b} \|u - \tu \|^2 + 4 \|g - \tg\|^2 \\
	\leq &\, - \frac{b}{2} \int_\gw (u - \tu)^2 (u^2 + \tu^2)\, dx + 4\left( 1+ \frac{a^2}{b}\right) \|\Pi (t) \|^2. 	
	\end{split}
\eeq
Similarly, for the second and third components of $\Pi(t) = g(t) - \tg (t))$, we get
\beq \label{Ve}
	\begin{split}
	\frac{d}{dt} &\, \|V(t)\|^2 \leq \frac{d}{dt} \|V(t)\|^2 + 2d_2\|\nb V(t)\|^2 \leq 2\langle \psi (u) - \psi (\tu) - (v - \tv), v - \tv \rangle \\[3pt]
	= &\, 2\int_\gw \left( - \beta (u^2- \tu^2) - (v  - \tv) \right) (v - \tv)\, dx \\
	\leq &\, 2\int_\gw \left(- \beta (u - \tu) u (v - \tv) - \beta (u - \tu)\tu (v - \tv)\right) dx \\
	\leq &\, \int_\gw \left(\frac{bu^2}{2} (u - \tu)^2 + \frac{b\tu^2}{2} (u - \tu)^2\right) dx + \frac{4\beta}{b} \|v - \tv \|^2  \\
	\leq &\; \frac{b}{2} \int_\gw \, (u^2 + \tu^2) (u - \tu)^2 \, dx + \frac{4\beta}{b}\, \|\Pi (t) \|^2
	\end{split}
\eeq
and
\beq \label{We}
	\begin{split}
	\frac{d}{dt} &\, \|W(t)\|^2 \leq \frac{d}{dt} \|W(t)\|^2 + 2d_3\|\nb W(t)\|^2 \leq 2\langle q(u - \tu) - r(w - \tw), w - \tw \rangle \\[3pt]
	= &\, 2\int_\gw \left(q (u- \tu) - r(w  - \tw) \right) (w - \tw)\, dx \\[3pt]
	\leq &\, q \|u - \tu\|^2 + (q + 2r) \|w - \tw \|^2 \leq 2 (q + r)\, \|\Pi(t) \|^2.
	\end{split}
\eeq
Sum up the inequalities \eqref{Ue}, \eqref{Ve} and \eqref{We} with a cancellation of the first terms on the rightmost side of \eqref{Ue} and \eqref{Ve}. Then we obtain
\beq \bl{vpp}
	\begin{split}
	&\frac{d}{dt} \|\Pi \|^2 + 2(d_1 \|\nb U \|^2 + d_2 \|\nb V \|^2 + d_3\|\nb W \|^2) = 2\langle f(g) - f(\tg), \Pi \rangle \\[3pt]
	\leq &\, \left(4\left[1 + \frac{1}{b} (a^2 + \beta)\right] + 2(q + r) \right)\|\Pi \|^2.
	\end{split}
\eeq
It follows that, for any $g_\tau, \tg_\tau \in M^*_E$ and indeed for any $g_\tau, \tg_\tau \in H$,
\beq \bl{piq}
	\frac{d}{dt} \, \|\Pi\|^2 \leq C_*\|\Pi \|^2
\eeq
where the constant $C_* = 4\left(1 + \frac{1}{b} (a^2 + \beta)\right) +2(q + r)$. Consequently,
\beq \bl{PiH}
	\begin{split}
	&\|S(t + \tau, \tau) g_\tau - S(t + \tau,\tau)\tg_\tau \|^2 = \|\Pi (t + \tau)\|^2 \\[6pt]
	\leq &\, e^{C_*t}\|\Pi (\tau)\|^2 = e^{C_*t}\|g_\tau-\tg_\tau\|^2, \quad t \geq 0, \;\, \tau \in \mathbb{R}.
	\end{split}
\eeq

Step 2. In oder to prove \eqref{lpHe}, we express the weak solution of \eqref{vpeq} by using the mild solution formula,
\beq \bl{milds}
	\Pi (t + \tau) = e^{At} \, \Pi (\tau) + \int_\tau^{t+\tau} e^{A(t +\tau - s)} (f(g(s)) - f(\tg (s))\,ds, \quad t \geq 0,
\eeq
where the $C_0$-semigroup $\{e^{At}\}_{t \geq 0}$ is generated by the operator $A$ defined in \eqref{opA}. By the regularity property of the analytic $C_0$-semigroup $\{e^{At}\}_{t \geq0}$ \cite{Rb,SY}, it holds that $e^{At}: H \to E$ for $t > 0$ and there is a constant $C_1 > 0$ such that
\beq \bl{eAt}
	\|e^{At}\|_{\mathcal{L}(H, \,E)} \leq C_1 \,  t^{-1/2}, \quad t > 0.
\eeq
Thus we have
\beq \bl{mdineq}
	\begin{split}
	\|\Pi (t + \tau)\|_E \leq &\, \|e^{At}\|_{\mathcal{L}(H, \,E)} \|\Pi (\tau)\| +  \int_\tau^{t+\tau} \|e^{A(t +\tau - s)}\|_{\mathcal{L}(H, E)}\| (f(g(s)) - f(\tg (s))\|\,ds \\[4pt]
	\leq &\, \frac{C_1}{\sqrt{t}} \,\|g_\tau - \tg_\tau\| + \int_{\tau}^{t + \tau} \frac{C_1}{\sqrt{t + \tau - s}} \, \| (f(g(s)) - f(\tg (s))\|\,ds, \quad t > 0.
	\end{split}
\eeq
Here we estimate the norm of the difference in the last integral of \eqref{mdineq},
\begin{align*}
	&\|f(g) - f(\tg)\|^2 = \|\vp (u) - \vp(\tu) + (v - \tv) -(w-\tw)\|^2 \\[3pt]
	+ &\, \|\psi (u) - \psi (\tu)- (v - \tv)\|^2 + \|q(u - \tu) - r(w-\tw)\|^2 \\[3pt]
	\leq &\, 3\|\vp (u) - \vp (\tu)\|^2 + 2\|\psi (u) - \psi (\tu)\|^2 + 2q \|u -\tu\|^2 + 5 \|v -\tv \|^2 + (3+2r) \|w-\tw\|^2  \\[3pt]
	= &\, (3a^2 + 2\beta^2)\|u^2 - \tu^2\|^2 + 3b^2 \|u^3 - \tu^3 \|^2 \\[3pt]
	+ &\, 2q \|u -\tu\|^2 + 5 \|v -\tv \|^2 + (3+2r) \|w-\tw\|^2 \\[3pt]
	\leq &\, (6a^2 + 4\beta^2)(\|u\|^2 +\|\tu\|^2)\|u - \tu\|^2 + 3b^2 \|u^2 + u\tu + \tu^2\|^2 \|u - \tu\|^2 \\[3pt]
	+ &\, 2q \|u -\tu\|^2 + 5 \|v -\tv \|^2 + (3+2r) \|w-\tw\|^2 \\[3pt]
	\leq &\, (6a^2 + 4\beta^2)(\|u\|^2 +\|\tu\|^2)\|u - \tu\|^2 + 3b^2 \|u^2 + u\tu + \tu^2\|^2 \|u - \tu\|^2 \\[3pt]
	+ &\, 2q \|u -\tu\|^2 + 5 \|v -\tv \|^2 + (3+2r) \|w-\tw\|^2,
\end{align*}
where in the term $3b^2 \|u^2 + u\tu + \tu^2\|^2 \|u - \tu\|^2$, we deduce that
\begin{align*}
	&\|u^2 + u\tu + \tu^2\|^2 = \int_\gw (u^2 + u\tu + u^2)^2 \, dx \\
	=&\, \int_\gw (u^4 + 3u^2\tu^2 + \tu^4 + 2u\tu (u^2 + \tu^2))\, dx \\
	\leq &\, \left(u^4 + \tu^4 + \frac{3}{2} (u^4 + \tu^4) + u^2 \tu^2 + (u^2 + \tu^2)^2 \right) dx \\
	\leq &\, 5\int_\gw (u^4 + \tu^4)\, dx = 5 \left(\|u\|^4_{L^4} + \|\tu\|^4_{L^4}\right). 
\end{align*}
Substitute the above inequalities into the integral term in \eqref{mdineq} and use the embedding inequality \eqref{L4H1} to obtain
\beq \bl{mdest}
\begin{split}
	&\|\Pi (t + \tau)\|_E \leq \frac{C_1}{\sqrt{t}} \, \|g_\tau - \tg_\tau\| + \int_{\tau}^{t + \tau} \frac{C_1}{\sqrt{t + \tau - s}}\, \| (f(g(s)) - f(\tg (s))\|\,ds \\
	\leq &\, \frac{C_1}{\sqrt{t}} \,\|g_\tau - \tg_\tau\| + \int_{\tau}^{t + \tau} \frac{C_1}{\sqrt{t + \tau - s}} \,(6a^2 + 4\beta^2)(\|u\|^2 +\|\tu\|^2)\|u - \tu\|^2\, ds \\
	+ &\,\int_\tau^{t+\tau}  \frac{C_1}{\sqrt{t + \tau - s}}\, 30b^2 \rho \left(\|u\|^4_{H^1} + \|\tu\|^4_{H^1}\right)\|u - \tu\|^2\,ds \\
	+ &\, \int_{\tau}^{t + \tau} \frac{C_1}{\sqrt{t + \tau - s}} \,( 2q \|u -\tu\|^2 + 5 \|v -\tv \|^2 + (3+2r) \|w-\tw\|^2 )\,ds, \: t \geq 0,
\end{split}
\eeq
for any $\tau \in \mathbb{R}$. 

Note that from \eqref{gest} and \eqref{acB}, since both $g_\tau$ and $\tg_\tau$ are in $M^*_E$, we have
\beq \bl{gbd}
	\|u(t+\tau)\|^2 \leq \|g(t+\tau)\|^2 \leq G_1 = \frac{\max \{1,c_1\}}{\min \{1, c_1\}} K_2 + K_1, \quad \text{for} \;\, t \geq 0,\tau \in \mathbb{R},
\eeq
where the positive constants $K_1$ and $K_2$ are given in \eqref{M1} and \eqref{acB} respectively, and independent of $t$ and $\tau$.

Step 3. We want to improve the inequality \eqref{acg}:
\begin{equation*}
\begin{split}
	&\frac{d}{dt} \left(\|\nb u \|^2 + \|\nb v\|^2 + \|\nb w\|^2\right) + d_1 \|\Delta u\|^2 + d_2 \|\Delta v\|^2 + d_3 \|\Delta w\|^2 \\[7pt]
	+ &\, 6b \|u \nb u\|^2 + 2\|\nb v \|^2 + 2r \|\nb w \|^2 \\
	\leq &\,  \frac{4}{d_1} \|v\|^2 + \frac{4}{d_1}\| w \|^2  + \frac{2q^2}{d_3} \|u\|^2 +  \left(\frac{4a^2}{d_1} + \frac{2\beta^2}{d_2}\right) \|u\|^4_{L^4} \\
	+&\, \left(\frac{8J^2}{d_1} + \frac{4\alpha^2}{d_2} + \frac{4q^2c^2}{d_3}\right) |\gw| + \frac{8}{d_1} \|p_1 (t)\|^2  + \frac{4}{d_2} \|p_2 (t)\|^2 + \frac{4}{d_3} \|p_3(t)\|^2.
\end{split}
\end{equation*} 
Specifically we need to further treat the following term on the right-hand side of the above \eqref{acg},
$$
	 \left(\frac{4a^2}{d_1} +\frac{2\beta^2}{d_2}\right) \|u\|^4_{L^4}
$$
by using the Gagliardo-Nirenberg inequality \eqref{GN} for the interpolation spaces
$$
	L^1(\gw) \hookrightarrow L^2 (\gw) \hookrightarrow H^1 (\gw).
$$
It implies that there is a constant $C > 0$ and
\beq \bl{L12}
   \|u^2\|^2 \leq C \|\nb (u^2)\|^{6/5}\|u^2\|^{4/5}_{L^1}.
\eeq
because $-\frac{3}{2} =\theta (1 - \frac{3}{2}) - 3(1- \theta)$ with  $\theta = 3/5$ and $1 - \theta = 2/5$. Therefore, the inequality \eqref{L12} and the Young's inequality imply that there exists a constant $0 < \ve < b$ such that
\beq \bl{u4}
	\begin{split}
	 &\left(\frac{4a^2}{d_1} + \frac{2\beta^2}{d_2}\right)\|u\|^4_{L^4} = \left(\frac{4a^2}{d_1} + \frac{2\beta^2}{d_2}\right) \|u^2\|^2 \\[2pt]
	 \leq &\, C\left(\frac{4a^2}{d_1} + \frac{2\beta^2}{d_2}\right)\|\nb (u^2)\|^{6/5}\|u^2\|^{4/5}_{L^1}\leq \ve \|\nb u^2\|^2 + C_\ve \|u^2\|_{L^1}^2 \\
	 = &\, 4\ve \,\|u \nb u\|^2 +C_\ve \|u\|^4 \leq 4b \,\|u \nb u\|^2 +C_\ve \|u\|^4,
	\end{split}
\eeq
where $C_\ve > 0$ is a constant only depending on $\ve$. Substitute \eqref{u4} into the above inequality \eqref{acg} to obtain
\begin{equation} \bl{nbg}
\begin{split}
&\frac{d}{dt} \left(\|\nb u \|^2 + \|\nb v\|^2 + \|\nb w\|^2\right) + d\left( \|\Delta u\|^2 + \|\Delta v\|^2 + \|\Delta w\|^2 \right) \\[9pt]
+ &\, 2b\, \|u \nb u\|^2 + 2\|\nb v \|^2 + 2r \|\nb w \|^2 \leq G_2 (\|u\|^2 +\|v\|^2 + \| w \|^2) \\[7pt]
+&\, G_3 |\gw| + \frac{8}{d} \left(\|p_1 (t)\|^2  + \|p_2 (t)\|^2 + \|p_3(t)\|^2 \right) + C_\ve \|u\|^4, \quad t \geq \tau \in \mathbb{R},
\end{split}
\end{equation} 
where $d = \min \{d_1, d_2, d_3\}$, 
$$
	G_2 = \frac{1}{d} \max \{4, 2q^2\} \quad \text{and} \quad G_3 = \frac{1}{d} \, (8J^2 + 4\alpha^2 + 4q^2c^2).
$$
Then the inequality \eqref{nbg} with \eqref{gbd} infers that
\beq \bl{gG}
\begin{split}
	\frac{d}{dt} \|\nb (u, v, w) \|^2 &\; \leq  G_2 \|(u, v, w)\|^2 + G_3 |\gw| + \frac{8}{d} \,\|p(t)\|^2 + C_\ve \|u\|^4\\[3pt] 
	&\; \leq  G_1G_2 + G_3 |\gw| +  \frac{8}{d} \,\|p(t)\|^2 + C_\ve K_1^2
\end{split}
\eeq 
for any $t \geq \tau \in \mathbb{R}$. It follows that for any $0 \leq t \leq T_{M^*_E}$ we have
\beq \bl{gB}\
	\begin{split}
    &\|u(t+\tau)\|^2_{H^1 (\gw)} = \|u(t+\tau)\|^2 + \|\nb u(t+\tau)\|^2  \\[5pt]
    \leq &\, \|u(t+\tau)\|^2 + \|\nb u(\tau)||^2 + \int_\tau^{\tau + t} \left(G_1G_2 + G_3 |\gw| +  \frac{8}{d}\, \|p(s)\|^2 + C_\ve K_1^2\right) ds \\
    \leq &\, G_1 + \|\nb u(\tau)\|^2 + t(G_1G_2 + G_3 |\gw| + C_\ve K_1^2) + \frac{8}{d} \int_\tau^{t+\tau} \|p(s)\|^2\, ds \\
    \leq &\, G_1 + K_2 + T_{M^*_E}(G_1G_2 + G_3 |\gw| + C_\ve K_1^2) + \frac{8}{d} \left(T_{M^*_E} + 1\right) \|p\|^2_{L^2_b}.
    \end{split}
\eeq

Step 4. Finally we substitute \eqref{gbd} and \eqref{gB} into the inequality \eqref{mdest} for any two solutions $g(t)$ and $\tg (t)$ of the nonautonomous system \eqref{napb} with initial states in $M^*_E$. Then for any $t > 0$ and $\tau \in \mathbb{R}$, it holds that
\beq \bl{PEst}
\begin{split}
	&\|\Pi (t + \tau)\|_E \leq \frac{C_1}{\sqrt{t}}\, \|g_\tau - \tg_\tau\| + \int_{\tau}^{t + \tau} \frac{C_1}{\sqrt{t + \tau - s}} \,\| (f(g(s)) - f(\tg (s))\|\,ds \\
	\leq &\, \frac{C_1}{\sqrt{t}} \,\|g_\tau - \tg_\tau\| + \int_{\tau}^{t + \tau} \frac{C_1}{\sqrt{t + \tau - s}} \, G_1(12a^2 + 8\beta^2)\|u - \tu\|^2\, ds \\
	+ &\,\int_\tau^{t+\tau}  \frac{C_1}{\sqrt{t + \tau - s}}\, 30b^2 \rho \left(\|u\|^4_{H^1} + \|\tu\|^4_{H^1}\right)\|u - \tu\|^2\,ds \\
	+ &\, \int_{\tau}^{t + \tau} \frac{C_1}{\sqrt{t + \tau - s}} \max \{2q, 5, 3+ 2r\}(\|u -\tu\|^2 + \|v -\tv \|^2 + \|w-\tw\|^2)\,ds \\
	\leq &\, \frac{C_1}{\sqrt{t}} \,\|g_\tau - \tg_\tau\| + \int_{\tau}^{t + \tau} \frac{C_1}{\sqrt{t + \tau - s}}\, G_p \,\|g(s) - \tg (s)\|^2\, ds \\[2pt]
	\leq &\, \frac{C_1}{\sqrt{t}} \,\|g_\tau - \tg_\tau\| + \int_{\tau}^{t + \tau} \frac{C_1}{\sqrt{t + \tau - s}}\, G_p \, e^{C_*(s-\tau)} \,\|g_\tau - \tg_\tau\|^2\, ds \\[2pt]
	\leq &\, \frac{C_1}{\sqrt{t}} \, \|g_\tau - \tg_\tau\| + \int_{\tau}^{t + \tau} \frac{C_1}{\sqrt{t + \tau - s}} \, G_p \, e^{C_*(s-\tau)} 2\sqrt{G_1} \|g_\tau - \tg_\tau\|\, ds 
\end{split}
\eeq
where we used \eqref{PiH} and \eqref{gbd} in the last two steps, and the positive constant $G_p$ is given by
\beq \bl{Gp}
	\begin{split}
	G_p = &\, G_1(12a^2 + 8\beta^2)+  \max \{2q, 5, 3+ 2r\} \\
	+ &\, 60 b^2 \rho \left[ G_1 + K_2 + T_{M^*_E}(G_1G_2 + G_3 |\gw|) + \frac{8}{d} (T_{M^*_E} + 1)\|p\|^2_{L^2_b}\right]^2 
	\end{split}
\eeq
which depends on the nonautonomous terms $p_i (t, x), i = 1, 2, 3$, and the permanently entering time $T_{M^*_E}$. Integrating the inequality \eqref{PEst} on the time interval $[\tau, \tau + T_{M^*_E}]$, here without of generality $T_{M^*_E} > 0$,  we then obtain the result that
\beq \bl{PiE}
	\begin{split}
	&\|S(\tau + T_{M^*_E}, \tau) g_\tau - S(\tau + T_{M^*_E},\tau) \tg_\tau \|_E = \|\Pi (\tau + T_{M^*_E})\|_E   \\
	\leq &\, C_1 \left(\frac{1}{\sqrt{T_{M^*_E}}} + 4\sqrt{G_1 T_{M^*_E}}\, \exp \left\{{C_* T_{M^*_E}}\right\} G_p \right) \|g_\tau - \tg_\tau\|
	\end{split}
\eeq
for any $g_\tau, \tg_\tau \in M^*_E$ and any $\tau \in\mathbb{R}$. Therefore, \eqref{lpHe} and then \eqref{lpH} are proved with the uniform Lipschitz constant
$$
	\kappa = C_1 \left(\frac{1}{\sqrt{T_{M^*_E}}} + 4\sqrt{G_1 T_{M^*_E}}\, \exp \left\{C_* T_{M^*_E}\right\} G_p \right).
$$ 
The proof of this theorem is completed.
\end{proof}

After the challenging Theorem \ref{LpHE} has been proved, now we can prove the second main result of this paper.

\begin{theorem} \bl{pbexA}
	For the nonautonomous Hindmarsh-Rose process $\{S(t, \tau)_{t \geq \tau \in \mathbb{R}}$ generated by the nonautonomous Hindmarsh-Rose equations \eqref{uq}-\eqref{wq}, there exists a pullback exponential attractor $\mathscr{M} = \{\mathscr{M}(\tau)\}_{\tau\in \mathbb{R}}$ in the space H.
\end{theorem}

\begin{proof}
	We can apply Proposition \ref{PeA} to prove this theorem. Indeed Lemma \ref{naac} and Theorem \ref{LpHE} have shown that the first two conditions in that Proposition \ref{PeA} are satisfied with the pullback absorbing set $M^* = M^*_E$ in \eqref{ab} by the nonautonomous Hindmarsh-Rose process $S(t, \tau)_{t \geq \tau \in \mathbb{R}}$. Thus it suffices to show that the third condition of \eqref{Lpt1} and \eqref{Lpt2} in Proposition \ref{PeA} is satisfied. 
	
        Recall that the Hindmarsh-Rose process $S(t, \tau)$ is defined by \eqref{CP} and let $g(t, \tau, g_\tau)$ be the weak solution to the initial value problem of the nonautonomous Hindmarsh-Rose evolutionary equation \eqref{napb}. For any $t_1 < t_2$ with $|t_1 - t_2| \leq I$, where $I$ is any given positive constant, we can estimate the $H$-norm of the difference of two pullback solution trajectories 
$$
	g^1 (t) = S(t, \tau - t_1)g_0 \quad \text{and} \quad g^2 (t) = S(t, \tau - t_2)g_0, \quad 0 \leq t_1 \leq t_2, \; g_0 \in H, 
$$
as follows. 

Using the notation in \eqref{vpeq} but here $\Pi (t) = g^1 (t) - g^2 (t)$. Then $\Pi (t)$ is the solution of the initial value problem
\begin{equation} \bl{Steq}
	\begin{split}
	\frac{d\Pi}{dt} &\, = A\Pi + f(g^1) -f(g^2), \quad t \geq \tau - t_1 \in \mathbb{R}, \\[3pt]
	&\Pi (\tau - t_1) = g_0 - S(\tau - t_1, \tau - t_2)g_0.	
	\end{split}
\end{equation}
By \eqref{piq}, we have
\beq \bl{Stpt}
	\begin{split}
	\frac{d\|\Pi\|^2}{dt} \leq C_* \|\Pi \|^2, \quad t \geq \tau,
	\end{split}
\eeq
where $C_*$ is the same constant as in \eqref{piq}. 

The Lipschitz and H\"{o}lder continuity associated with the regularity property of the parabolic $C_0$-semigroup of contraction $\{e^{At}\}_{t \geq 0}$, cf. \cite{SY}, gives rise to
\beq \bl{LH}
	\|e^{A(t_0 + h)} g_0 - e^{At_0} g_0\| \leq \|e^{At_0} \| \|e^{Ah} g_0 - g_0\| \leq C_0 |h| \|g_0\|, \quad \text{for all} \;\; t_0 \geq 0,
\eeq
where $C_0 > 0$ is a constant depending only on the contraction operator semigroup $e^{At}$. Then, it follows from \eqref{Stpt} and \eqref{LH} that 

\beq \bl{St12}
	\begin{split}
	&\|S(t, \tau - t_1) g_0 - S(t, \tau - t_2)g_0\| = \|g^1 (t, \tau - t_1, g_0) - g^2 (t, \tau - t_2, g_0)\| \\[2pt]
	= &\, \left\|e^{A(t - (\tau - t_1))}g_0 + \int_{\tau - t_1}^t e^{A(t-s)} [f(g^1 (s, \tau - t_1, g_0)) + p(s, x)]\, ds \right.  \\
	- &\, \left. e^{A(t - (\tau - t_2))}g_0 - \int_{\tau - t_2}^t e^{A(t-s)} [f(g^2 (s, \tau - t_2, g_0)) + p(s, x)]\, ds \right\| \\
	\leq &\, \|\Pi (t, \tau - t_1, \Pi (\tau - t_1)\| \leq e^{\frac{1}{2}C_* |t - (\tau - t_1)|} \|\Pi (\tau - t_1, \tau - t_2, g_0) \| \\[6pt]
	\leq &\, e^{\frac{1}{2}C_* |t - (\tau - t_1)|} \, \|e^{A(t_2 - t_1)}g_0 -  g_0)\|  \\[2pt]
	+ &\, e^{\frac{1}{2}C_* |t - (\tau - t_1)|} \int_{\tau - t_2}^{\tau - t_1} \| e^{A(\tau - t_1-s)} [f(g^2 (s, \tau - t_2, g_0)) + p(s, x)]\|\, ds  \\
	\leq &\, e^{\frac{1}{2}C_* |t - (\tau - t_1)|} \, C_0 \, |t_1 - t_2| \|g_0\| \\
	+ &\, e^{\frac{1}{2}C_* |t - (\tau - t_1)|}  \int_{\tau - t_2}^{\tau - t_1} \|e^{A(t-s)} [f(g^2 (s, \tau - t_2, g_0)) + p(s, x)]\|\, ds.
	\end{split}
\eeq
Denote by $T^* = T_{M^*_E} > 0$, which is the finite time when all the pullback solution trajectories started from the pullback absorbing set $M_E^*$ in Lemma \ref{naac} permanently enter into itself. Define the following set, where the closure is taken in the space $E$,
\beq \bl{GB}
	\Gamma = \overline{\bigcup_{0 \leq t \leq T^*} S(\tau, \tau - t) M_E^*}
\eeq
Lemma \ref{naac} demonstrated that $M_E^*$ and $T^*$ are independent of $\tau \in \mathbb{R}$ and $t \geq 0$. Denote by $D_\Gamma =  \max_{g \in \Gamma} \|f(g) \|_H$, since the Nemytskii operator $f: E \to H$ is bounded on the bounded set $\Gamma$ in $E$. Here $\|e^{At}\|_{\mathcal{L}(H)} \leq 1$ and by H\"{o}lder inequality, 
\beq \bl{fg2}
	\begin{split}
	 &\, \int_{\tau - t_2}^{\tau - t_1} \|e^{A(t-s)}\|_{\mathcal{L}(H)} (\| f(g^2 (s, \tau - t_2, g_0)) \| + \| p(s, x)\|)\, ds \\
	 \leq &\, (D_\Gamma + K_2) |t_1 - t_2| + \int_{\tau - t_2}^{\tau - t_1} \| p(s, \cdot )\|\, ds \\
	 \leq &\, (D_\Gamma + K_2) |t_1 - t_2| + |t_1 - t_2|^{1/2} \sqrt{\int_{\tau - t_2}^{\tau - t_1} \|p(s, \cdot )\|^2\, ds}    \\
	 \leq &\, (D_\Gamma + K_2) |t_1 - t_2| + |t_1 - t_2|^{1/2} \sqrt{(|t_1 - t_2| +1) \, \Sigma_{i=1}^3 \|p_i \|^2_{L_b^2}} \\[2pt]
	 \leq &\, (D_\Gamma + K_2) |t_1 - t_2| + |t_1 - t_2|^{1/2} ( |t_1 - t_2|^{1/2} + 1) \|p \|_{L_b^2} \\[6pt]
	 \leq &\,( |t_1 - t_2| + |t_1 - t_2|^{1/2}) (D_\Gamma + K_2 + \|p \|_{L_b^2}), 
	\end{split}
\eeq
for any $t_1 \geq T^*$ and $g_0 \in M_E^*$, where $K_2$ is given in \eqref{acB}. 

Substituting \eqref{fg2} into \eqref{St12} we obtain
\beq  \bl{G12}
	\begin{split}
	&\| S(t, \tau - t_1) g_0 - S(t, \tau - t_2) g_0\| =  \|g^1 (t, \tau - t_1, g_0) - g^2 (t, \tau - t_2, g_0) \| \\[5pt]
	\leq &\, e^{\frac{1}{2}C_* |t - (\tau - t_1)|} C_0 \, |t_1 - t_2| \|g_0\| \\[5pt]
	 +  &\,  e^{\frac{1}{2}C_* |t - (\tau - t_1)|} ( |t_1 - t_2| + |t_1 - t_2|^{1/2}) (D_\Gamma + K_2 + \|p \|_{L_b^2}) \\[5pt]
	\leq &\, \lambda (M_E^*)\, e^{\frac{1}{2}C_* |t - (\tau - t_1)|}  |t_1 - t_2|^\gamma, \quad \text{for} \;\, t \geq \tau - t_1,  \; t_1 \geq T^*, \; g_0 \in M_E^*.
	\end{split}
\eeq	
where
$$
	\lambda (M_E^*) = C_0 K_2 + 2(D_\Gamma + K_2 + \|p \|_{L^2_b})
$$
and 
$$
	\gamma = \begin{cases}
			\frac{1}{2}, &\text{if $|t_1 - t_2| < 1$;} \\[5pt]
			1, &\text{if $|t_1 - t_2| \geq 1$.}
	 		\end{cases}
$$

For any given $\tau \in \mathbb{R}$, in the above inequality \eqref{G12} take 
$$
	t = \tau, \quad t_1 = T_{M^*_E}, \quad \text{and} \quad t_2 = T_{M^*_E} + t \quad  \text{for} \; \;t \in [0, \,T^*_{M^*_E}].
$$
Then we obtain
\beq \bl{C59}
	\sup_{\tau \in \mathbb{R}} \, \|S(\tau, \tau - T_{M^*_E}) g_0 - S(\tau, \tau - T_{M^*_E} - t )g_0\| \leq \lambda (M_E^*)\, \exp \left\{\frac{C_*}{2}\,T_{M^*_E}\right\} | t |^\gamma 
\eeq
for $t \in [0, T_{M^*_E}], \, g_0 \in M^*_E$. It shows that the Lipschitz condition \eqref{Lpt1} with $M^* = M^*_E$ in Proposition \ref{PeA} is satisfied. Moreover, for any given $\tau \in \mathbb{R}$, take $t = \tau$ and $t_1, t_2 \in [T_{M^*_E}, 2T_{M^*_E}]$ in \eqref{G12}, we see that 
\beq \bl{f510}
	\|S(\tau, \tau - t_1) g_0 - S(\tau, \tau - t_2)g_0\| \leq \lambda (M_E^*)\, \exp \left\{\frac{C_*}{2}\,T_{M^*_E}\right\}\, |t_1 - t_2|^\gamma
\eeq
for any $g_0 \in M^*_E$. It shows that the Lipschitz condition \eqref{Lpt2} with $M^* = M^*_E$ is also satisfied by the nonautonomouss Hindmarsh-Rose process. According to Proposition \ref{PeA}, there exists a pullback exponential attractor  $\mathscr{M} = \{\mathscr{M}(\tau)\}_{\tau\in \mathbb{R}}$ in the space $H$. The proof of this theorem is completed.
\end{proof}

\bibliographystyle{amsplain}

\end{document}